\documentclass{amsart}
\usepackage{amssymb}
\usepackage{amsmath}
\usepackage{latexsym}
\input{xy}
\xyoption{all}

\newtheorem{theorem}{Theorem}[section]
\newtheorem{lemma}[theorem]{Lemma}

\newtheorem*{varthmC}{Conjecture}
\newtheorem*{varthm1}{Main Theorem}

\newtheorem{Def}[theorem]{Definition}
\newtheorem*{KC}{Korselt's Criterion}

\begin{document}
\title[Carmichael Numbers in Arithmetic Progressions]{Infinitely Many Carmichael Numbers in Arithmetic Progressions}
\author{Thomas Wright}
\maketitle

\begin{abstract}
In this paper, we prove that for any $a,M\in \mathbb N$ with $(a,M)=1$, there are infinitely many Carmichael numbers $m$ such that $m\equiv a$ mod $M$.\end{abstract}

\section{Introduction}
Fermat's Little Theorem asserts that if $m$ is prime then
\[a^m\equiv a\pmod m.\]
In 1910, investigations by R.D. Carmichael [Ca] into whether the converse of Fermat's Little Theorem is true unearthed the following objects:

\begin{Def} Let $m \in \mathbb N$.  If $m|a^m-a$
for every $a \in \mathbb Z$ and $m$ is not prime then $m$ is a
\underline{Carmichael Number}.\end{Def}

Although Carmichael numbers would seem at first glance to be extremely rare (there are only seven Carmichael numbers less than 10,000), it was nevertheless proven by Alford, Granville, and Pomerance in 1994 that there are infinitely many [AGP]; in fact, subsequent results have proven a lower bound of roughly $x^{1/3}$ for the density of this set [Ha2], and it has been conjectured that the density is actually $x^{1-o(1)}$ (see [Er], [GP]).

After the proof of the infinitude of Carmichael numbers in [AGP], focus began to turn to natural follow-up questions, including the occurrence of Carmichael numbers in arithmetic progressions.  To this end, the following conjecture was made:

\begin{varthmC}  Let $M$ be a positive integer, and let $(a,M)=1$.  Then there are infinitely many Carmichael numbers $m$ such that

\[m\equiv a\pmod M.\]
\end{varthmC}


Although this conjecture has been open for some time, it has not been without progress.  For instance, the original proof of infinitely many Carmichael numbers is easily altered to prove the case where $m\equiv 1$ (mod $M$).  Banks and Pomerance [BP] gave a conditional proof of the above conjecture in 2010, although their proof required the assumption of a very strong (stronger than GRH) conjecture on the size of the first prime in an arithmetic progression.  In 2012, Matom\"{a}ki [Ma] gave an unconditional proof for the special case where $a$ is a quadratic residue mod $M$.  However, the general case remained open; in fact, for a given $M$, it has not previously been known whether there are infinitely many Carmichael numbers that are not quadratic residues modulo $M$.

In this paper, however, we resolve this conjecture completely.  In particular, we have the following:



\begin{varthm1}  Let $M$ be a positive integer, and let $(a,M)=1$.  Then there are infinitely many Carmichael numbers $m$ such that

\[m\equiv a\pmod M.\]
Specifically, if we let $C_{M,a}(X)$ denote the number of Carmichael numbers up to $X$ that are congruent to $a$ mod $M$, there exists a constant $K>0$ for which

\[C_{M,a}(X)\gg X^{\frac{K}{(\log\log\log X)^2}}.\]
\end{varthm1}
This constant $K$ is explicitly computable (although it depends on $M$).  We do not engage in the computation of $K$ here for two reasons: doing so would make the paper somewhat less transparent, and it seems likely that our bound is well short of the actual lower bound for this quantity.

As is the case with nearly every paper written about the infinitude of Carmichael numbers since 1994, the proof in the current paper generally follows the outline of [AGP].  In particular, we begin by creating an $L$ with many factors $d$; we then show that for these $d$, there are many primes of the form $dk+1$ for $k$ relatively prime to $L$.  After finding a single $k$ which yields sufficiently many of these types of primes, we show that some subset of our primes can be combined to form Carmichael numbers.  Matom\"{a}ki [Ma] was the first to demonstrate how one can use this framework to prove properties about Carmichael numbers in arithmetic progressions; the present paper capitalizes on the alterations to the traditional framework that were introduced in that paper, using these new methods and adapting them to prove our results.



In the current paper, however, we introduce one major change from other methods: we make additional requirements on our primes $p$ in relation to $L$.  In particular, we require that each of our primes $p$ be a quadratic residue mod $L$.  Of course, the $p$ will generally be of the form $p=dk+1$ for some $d|L$, so this requirement of $p$ being a quadratic residue mod $L$ does not run afoul of the other requirements on $p$.

Alford, Granville, and Pomerance [AGP] make use of the fact that, apart from cases of exceptional moduli, the number of primes less than $x$ in a congruence class modulo $c$ can be estimated using classical estimates as long as $c\leq x^{1/B}$.  Unfortunately, since our primes $p$ must now fulfil requirements mod $L$, we require that $x$ be larger than $(ML')^{2/B}$ (for an $L'$ slightly larger than $L$), which is much a greater $x$ than is required in the original paper.  It should be noted that there is actually one benefit to this increased $x$.  Specifically, at a key point in the original proof of [AGP], there is sum over all the divisors $d$ of $L$ with $d\leq x$; in our paper, this is simply the sum of all divisors of $L$, which means that we can give an exact evaluation, rather than an estimate, for this quantity.



Unfortunately, the downside of using the larger $x$ is that we end up with a rather low asymptotic for the number of Carmichael numbers up to $X$; [Ma] shows that in the case where $a$ is a quadratic residue modulo $M$, the density is $\gg x^{1/5}$, and there does not seem to be any heuristic justification that a similar lower bound should not apply here.  In fact, it is entirely possible that one could find a lower bound $\gg x^{1-o(1)}$ (as has been conjectured for the number of Carmichael numbers themselves), but this sort of result is probably a long way off.

The primary change that we introduce (requiring that all of the primes be quadratic residues mod $L$) may be subtle, but it seems likely that this new technique can help with other problems related to Carmichael numbers.  For example, an open problem mentioned in [AGP] is to determine whether, for fixed integers $a$ and $b$, there are infinitely many natural numbers $n$ for which $p|n$ implies $p-a|n-b$.  This problem relates to Carmichael numbers via Korselt's Criterion, which states the following:
\begin{KC}  A natural number $n$ is a Carmichael number if and only if $n$ is square-free and for every prime $p$ that divides $n$, $p-1|n-1$.\end{KC}
As such, it is obvious that [AGP] resolves the case when $a=b=1$.  The generalization to other $a$ and $b$, however, is still an open problem, and this problem has been described as having ``significance for variants of pseudoprime tests."  The current paper can be easily altered to resolve this problem in the affirmative in the additional case of $a=b=-1$.

One could also use our method to show that for any $a,M\in \mathbb N$, there are infinitely many Carmichael numbers $m$ such that the number of prime factors of $m$ is congruent to $a$ (mod $M$).

We take these and other issues up in future papers (see [Wr2], [Wr3]).


\section{Primes in Arithmetic Progressions}


Our first step will be to identify the primes that we will use to construct the $L$ mentioned above.  Let $1<\theta<2$, and let $P(q-1)$ be the size of the largest prime divisor of $q-1$.  We then define the set $\mathcal Q$ by
\[\mathcal Q=\{q\mbox{ }prime:\frac{y^\theta}{\log y}\leq q\leq y^{\theta},\mbox{ }q\nmid M,\mbox{ }q\equiv -1\pmod{4\phi(M)},\mbox{ }P(q-1)\leq y\}.\]

We will require that $|\mathcal Q|$ be reasonably close to $\pi(y^\theta)$.  To this end, we prove the following:

\begin{lemma}\label{lemtwo}
For $\mathcal Q$ as above, there exist constants $\gamma=\gamma(\theta,M)$ and $Y_{\theta,\phi(M)}$ such that

\[|\mathcal Q|\geq \gamma \frac{y^{\theta}}{\log (y^\theta)}\]
if $y>Y_{\theta,\phi(M)}$
\end{lemma}
\begin{proof}  This proof appears in [Ma]; we replicate it here.

For $v<z$, let us denote by $\pi_{d,b}(z,v)$ the number of primes $q$ less than $z$ such that $P(q-1)\leq v$ and $q\equiv b$ (mod $d$).  Let $\frac 12<\alpha<\frac 23$, and define $\epsilon=\epsilon(\alpha)<\alpha-\frac 12$.  Note that if $q\leq z$ is such that $q$ can be written as $q=1+q'k$ for some prime $q'\in [z^{1-\alpha},z^{\frac 12-\epsilon}]$ then $P(q-1)\leq z^\alpha$; each $q$ has at most two such representations.  So
\[ \pi_{d,b}(z,z^{\alpha})\geq \frac 12\sum_{q'\in \mathbb P,\mbox{ }z^{1-\alpha}\leq q'\leq z^{\frac 12-\epsilon}}\#\{q\mbox{ }prime,\mbox{ }\frac{z}{\log z}\leq q\leq z,\mbox{ }q\equiv 1\pmod{q'},\mbox{ }q\equiv b\pmod d\}.\]
Since $d$ is fixed and $q$ is sufficiently large relative to $d$ and $q'$, we can consolidate our requirements on $q$ to be a single congruence modulo $dq'$, and hence we may use Bombieri-Vinogradov to find that
\[ \pi_{d,b}(z,z^{\alpha})\geq \sum_{q'\in \mathbb P,\mbox{ }z^{1-\alpha}\leq q'\leq z^{\frac 12-\epsilon}}\frac{z}{8\phi(dq')\log z}\geq \log\left(\frac{\frac 12-\epsilon}{1-\alpha}\right)\frac{z}{8\phi(d)\log z}.\]
The lemma then follows by letting $d=4\phi(M)$, $b=-1$, $z=y^\theta$, $\alpha=min\{\frac{1}{\theta},\frac 35\}$, and $\gamma=\frac{1}{8\phi(d)}\log\left(\frac{\frac 12-\epsilon}{1-\alpha}\right)$.
\end{proof}

The next step is to construct an $L$ and prove that there are a large number of primes that satisfy a variety of congruence conditions modulo $L$ and $M$.  However, we must first deal with a technicality involving the densities of primes in specific arithmetic progressions.

In particular, fix $B$ such that $0<B<5/12$.  Theorem 2.1 of [AGP] says that for any $x$ there exists a set of integers $\mathcal D_B(x)$, where $|\mathcal D_B(x)|$ is bounded by some constant $D_B$ and every integer in $\mathcal D_B(x)$ is greater than $\log x$, such that if $d$ is not divisible by an element in $\mathcal D_B(x)$ and $d\leq \min\{x^B,z/x^{1-B}\}$ then
\[\pi(z,d,c)\geq \frac{\pi(z)}{2\phi(d)},\]
for any $c$ with $(c,d)=1$.  Since we will wish to use this density estimate, we must be careful that our moduli are not divisible by an element of $\mathcal D_B(x)$.

In the original proof of Alford, Granville, and Pomerance [AGP], the authors apply this density estimate to find a lower bound for the number of primes up to $x$ that are 1 modulo the various divisors $d$ of $L$.  For this, the authors can work around this issue of problematic moduli by simply ignoring those divisors $d$ that are divisible by elements of $\mathcal D_B(x)$ and then invoking the density theorem on the remaining divisors; since $|\mathcal D_B(x)|$ is bounded, this still leaves many $d$ from which one can find primes as required.  Unfortunately, the luxury of simply ignoring inconvenient $d$ is not afforded to us in the present case, as we will also be requiring that all of our primes be quadratic residues modulo $L$, which means that we must be certain that $L$ is not divisible by any element of $\mathcal D_B(x)$.  To ensure that no member of $\mathcal D_B(x)$ divides $L$, we will first create an $L'$ for the purposes of determining our $x$; after this, we will remove from our $L'$ any multiples of elements in $\mathcal D_B(x)$, leaving us with an $L$ for which our density theorem applies.

We begin first with our definition of $L'$.  Let
\[L'=\prod_{q\in \mathcal Q}q.\]
For fixed $B$ with $0<B<5/12$, we then choose $x$ to be
\[x=\lceil \left(ML'\right)^{\frac{2}{B}}\rceil.\]
Now, for $\mathcal D_B(x)$ as described above, we can choose a set of primes $P_{B}(x)$, where  $|P_{B}(x)|\leq D_B$, such that any element in $\mathcal D_B(x)$ is divisible by at least one of the primes in $P_B(x)$.  From this, let
\[L=\prod_{q\in \mathcal Q,\mbox{ }q\not\in P_{B}(x)}q.\]
By this definition, no factor of $L$ is divisible by an element in $\mathcal D_B(x)$.  Moreover, since every element of $\mathcal D_B(x)$ is of size at least $\log x$, the assumption that $y$ is sufficiently large relative to $M$ implies that no divisor of $M$ is in $\mathcal D_B(x)$.

As mentioned above, we wish to show that there are a large number of primes less than $x$ that are 1 modulo divisors $d$ of $L$ while also being quadratic residues mod $L$ and congruent to $a$ modulo $M$.  It is worth noting that $(M,L)=1$ (since we may assume that $y$ is large relative to $M$), which ensures that these requirements on congruence classes do not contradict one another.

By analogy with the convention for denoting primes in arithmetic progressions, let us use the notation $\pi(z,d,\underline{QR})$ to indicate the number of primes up to $z$ that are quadratic residues mod $d$. Then we can prove the following:
\begin{lemma}  For any $a$ with $(a,M)=1$ and any $z\geq x^{1-\frac B2}$,
\[\pi(z,L,\underline{QR})\cap \pi(z,M,a)\geq \frac{z}{2^{\omega(L)+1}\phi(M)\log z},\]
where $\omega(L)$ denotes the number of prime factors of $L$.
\end{lemma}
\begin{proof} Clearly, the number of congruence classes mod $M$ is $\phi(M)$.  Moreover, if we work mod $L$, we note that the number of congruence classes that are quadratic residues modulo each prime $q|L$ is exactly $\frac{q-1}{2}$ of the $q-1$ classes which can contain prime numbers.  By Chinese Remainder Theorem, this means that the number of congruence classes mod $L$ that are quadratic residues is exactly $\prod_{q|L}(\frac{q-1}{2})$ of the $\prod_{q|L}q-1=\phi(L)$ congruence classes which yield a prime.


%
Now, since $ML\leq \min\{x^{B},z/x^{1-B}\}$ and $ML$ has no factors in $D_B$, we can apply Theorem 2.1 from [AGP] to congruence classes modulo $ML$.
In particular, we know that in the $\phi(M L)$ applicable congruence classes modulo $ML$, the number of primes in any one of these class will be
\[\geq \frac{\pi(z)}{2\phi(M L)}.\]
Since there are $\prod_{q|L}(\frac{q-1}{2})$ of these classes which would yield a quadratic residue mod $L$, we have that the number of primes that are $a$ mod $M$ and a quadratic residue mod $L$ are
\[\geq \frac{\pi(z)}{2^{\omega(L)+1}\phi(M)},\]
which is as required.\end{proof}

By analogy to the original paper of [AGP], for a given integer $d|L$ with $1\leq d\leq x^\frac B2$ for some fixed $B>0$, we wish to count the number of primes $p$ for which $p\equiv 1$ (mod $d$) and $((p-1)/d,L)=1$.  However, since we chose $x$ to be $\lceil(ML')^{\frac{2}{B}}\rceil$, every divisor $d$ of $L$ will be $\leq x^\frac B2$.  As such, we find the following:
\begin{lemma}\label{l512} Let $B<5/12$, and let $L$ be as above.  Then there exists a $k\leq x^{1-\frac B2}$ with $(k,L)=1$ such that
\begin{align*}
\#\{d|L &:p=dk+1\mbox{ }is\mbox{ }prime,\mbox{ }p\mbox{ }is\mbox{ }a\mbox{ }QR\mbox{ }mod\mbox{ }L,\mbox{ }p\equiv a\pmod M,\mbox{ }p\leq x\}\\
&\geq \frac{\left(\frac 32\right)^{\omega(L)}}{4\cdot \phi(M)\log x}.
\end{align*}

\end{lemma}
\begin{proof}
Above, we proved that for $z\geq x^{1-\frac B2}$,
\[\pi(z,L,\underline{QR})\cap \pi(z,M,a)\geq \frac{z}{2(2^{\omega(L)})\phi(M)\log z}.\]
If we add the additional constraint that the prime $p$ also be 1 mod $d$ for a given $d|L$, we have
\[\pi(dx^{1-\frac B2},d,1)\cap \pi(dx^{1-\frac B2},L,\underline{QR})\cap \pi(dx^{1-\frac B2},M,a)\geq \frac{dx^{1-\frac B2}}{2\cdot 2^{\omega(L)-\omega(d)}\phi(M)\phi(d)\log x},\]
where the savings of $2^{\omega(d)}$ in the denominator comes from the fact that we no longer have to worry whether $p$ is a quadratic residue mod $d$, and hence the requirement that $p$ be a quadratic residue mod $L$ is satisfied if $p$ is a quadratic residue mod $\frac Ld$.

We must now determine how many of these primes satisfy the additional condition of $((p-1)/d,L)=1$.  We require the technical condition that $\sum_{q|L}\frac{1}{q-1}\leq \frac{1}{64}$; however, just as in [AGP], this is easily verified for the $L$ we have chosen.

Now, for any prime $q$ which divides $L$, we have (by Montgomery and Vaughan's explicit version of the Brun-Titchmarsh theorem [MV]) that
\begin{align*}
\pi(dx^{1-\frac B2},dq,1)&\cap \pi(dx^{1-\frac B2},L,\underline{QR})\cap \pi(dx^{1-\frac B2},M,a)\\
\leq & \frac{2dx^{1-\frac B2}}{2^{\omega(L)-\omega(d)-1}\phi(M)\phi(dq)\log (x^{1-\frac B2}/q)}\\
\leq & \frac{8}{(q-1)(1-\frac B2)}\frac{dx^{1-\frac B2}}{2^{\omega(L)-\omega(d)}\phi(M)\phi(d)\log x},
\end{align*}
since $q\leq x^{(1-\frac B2)/2}$ by construction.  Putting these together, we have that
\begin{align*}
\pi(dx^{1-\frac B2},d,1)&\cap \pi(dx^{1-\frac B2},L,\underline{QR})\cap \pi(dx^{1-\frac B2},M,a)\\
&-\sum_{q|L,\mbox{ }q\mbox{ }prime}\pi(dx^{1-\frac B2},dq,1)\cap \pi(dx^{1-\frac B2},L,\underline{QR})\cap \pi(dx^{1-\frac B2},M,a)\\
\geq &\frac{dx^{1-\frac B2}}{2\cdot 2^{\omega(L)-\omega(d)}\phi(M)\phi(d)\log x}-\sum_{q|L,\mbox{ }q\mbox{ }prime}\frac{8}{(q-1)(1-\frac B2)}\frac{dx^{1-\frac B2}}{2^{\omega(L)-\omega(d)}\phi(M)\phi(d)\log x}\\
\geq &\frac{x^{1-\frac B2}}{4\cdot 2^{\omega(L)-\omega(d)}\phi(M)\log x},
\end{align*}
where the final step is by the fact that $\sum_{q|L}\frac{1}{q-1}\leq \frac{1}{64}$.  This means that we have at least
\[\sum_{d|L}\frac{x^{1-\frac B2}}{4\cdot 2^{\omega(L)-\omega(d)}\phi(M)\log x}\]
pairs $(p,d)$ such $p$ is prime, $d|L$, $(p-1)/d=k$ is relatively prime to $L$, $p$ satisfies the required congruences ($a$ mod $M$ and \underline{QR} mod $L$), $p\leq dx^{1-\frac B2}$, and $d\leq x^\frac B2$.

Since the number of possible distinct values for $(p-1)/d$ is bounded by $x^{1-\frac B2}$, there must exist some value $k$ such that $(k,L)=1$ and $k$ has at least
\[\sum_{d|L}\frac{2^{\omega(d)}}{4\cdot 2^{\omega(L)}\phi(M)\log x}\]
representations as $(p-1)/d$ for $p,d$ as above.

The numerator can be evaluated by a standard combinatorial identity:
\[\sum_{d|L}2^{\omega(d)}=\sum_{i=0}^{\omega(L)}\left(\begin{array}{c} \omega(L) \\ i \end{array} \right) 2^{\omega(L)-i}=(2+1)^{\omega(L)}=3^{\omega(L)},\]
and hence
\[\sum_{d|L}\frac{2^{\omega(d)}}{4\cdot 2^{\omega(L)}\phi(M)\log x}=\frac{\left(\frac 32\right)^{\omega(L)}}{4\cdot \phi(M)\log x}.\label{eqn1}\]\end{proof}

Let $k_0$ be the $k$ found above for which there exist many primes of the form $dk+1$ for $d|L$. Let us define
\[\mathcal P=\{p\mbox{ }prime:p=dk_0+1\mbox{ }for\mbox{ }some\mbox{ }d|L,\mbox{ }p\mbox{ }is\mbox{ }a\mbox{ }QR\mbox{ }mod\mbox{ }L,\mbox{ }p\equiv a\pmod M,\mbox{ }p\leq x\}.\]
The above lemma then gives a lower bound for the size of $|\mathcal P|$.

\section{Size of $\lambda(G)$}
Now, let $L$ be as above.  As usual, for a group $G$, we define $\lambda(G)$ to be the largest order of an element in $G$; in our case, we have $G=(\mathbb Z/ML\mathbb Z)^\times$.  Since $\lambda(G)$ is free of prime factors of size greater than $y$ (by assumption on the types of primes that divide $L$), we know that if $q^{a_p}|\lambda(G)$ then  $q\leq y$ and $q^{a_q}\leq y^\theta$.  Thus, if we let $a_q$ be the largest power of $q$ such that $q^{a_q}\leq y^\theta$ then
\[\lambda(G)\leq \prod_{q\leq y}q^{a_q}\leq y^{\theta \pi(y)}\leq e^{2\theta y}.\]

In the next section, we will compare this quantity to the number of primes generated by our choice of $L$.

\section{Sizes of Other Important Quantities}
In this section, we make use of a theorem from Baker and Schmidt [BS, Proposition 1], although the exact formulation of it comes from [Ma].  To begin, for an abelian group $G$, $n(G)$ is defined to be the smallest number such that a collection of at least $n(G)$ elements must contain some subset whose product is the identity.  Based upon the work of van Emde Boas and Kruyswijk [EK] and Meshulam [Me], it is known that
\[n(G)\leq \lambda(G)\left(1+\frac{\log|G|}{\lambda(G)}\right).\]
With this notation, we now state the theorem.
\begin{theorem}\label{thm}
For any multiplicative abelian group $G$, write
\[s(G) = \lceil 5\lambda(G)^2 \Omega (\lambda(G)) \log(3\lambda(G)
\Omega (|G|))\rceil,\]
where $\Omega(|G|)$ indicates the number of prime divisors (up to multiplicity) of $|G|$.

Let $A$ be a sequence of length $n$ consisting of non-identity elements of $G$. Then there exists a non-trivial subgroup $H\subset G$ such that\\

(i) If $n\geq s(G)$, then, for every $h\in H$, $A\cap H$ has a subsequence whose product is $h$.

(ii) If $t$ is an integer such that $s(G) < t < n-n(G)$ then, for every $h \in H$, $A$ has at least $\left(\begin{array}{c} n-n(G)\\ t-n(G)\end{array}\right)/\left(\begin{array}{c} n\\n(G)\end{array}\right)$
distinct subsequences of length at most $t$ and at least
$t - n(G)$ whose product is $h$.
\end{theorem}
\begin{proof}
See Lemma 6 of [Ma].

\end{proof}
For the rest of the paper, we will use this result in the case of $G=(\mathbb Z/ML\mathbb Z)^\times$ and $A=\mathcal P$.  To do this, we must bound the quantities $n$ and $s$ in terms of $y$ and $\theta$:
\begin{lemma}\label{lemfour2}  For $s(G)$ and $n(G)$ as defined above, $G=(\mathbb Z/ML\mathbb Z)^\times$, and $y$ sufficiently large,
\[s(G)\leq e^{7\theta y},\]
\[n(G)\leq e^{3\theta y}.\]
\end{lemma}
\begin{proof}
Follows from the fact that $\lambda(G)\leq e^{2\theta y}$ and $M$ is a constant.
\end{proof}
We require one more piece of information about the size of one of our variables before completing the proof of the main results.  Specifically, we would like to know the size of $L$ itself.  To this end, we have the following:
\begin{lemma}\label{lemfour}
For $y$ sufficiently large, there exist constants $0<\kappa_1<\kappa_2$ such that
\[e^{\kappa_1 y^\theta}\leq L\leq e^{\kappa_2 y^\theta}.\]
\end{lemma}
\begin{proof}
The upper bound of the product of the primes up to $x$ is well known to be less than $e^{1.02 x}$ (see [RS]).  We can thus choose $\kappa_2$ to be 1.02.

For the lower bound, define $E_\mathcal Q(n)$ to be 1 if $n\in \mathcal Q-D_B(x)$ and 0 otherwise.  Note that
\[\log L=\sum_{n=\frac{y^\theta}{\log y}}^{y^\theta}E_\mathcal Q(n)\log n\geq \left|\mathcal Q-D_B(x)\right|\log\left(\frac{y^\theta}{\log y}\right)\geq \frac{\gamma}{2}y^\theta.\]

\end{proof}
\section{Proof of Main Theorem}
\begin{theorem}
Let $\mathcal P$ be the set of primes defined at the end of Section 2, let $G=(\mathbb Z/ML\mathbb Z)^*$, and let $s(G)$ be as in Theorem \ref{thm}.  Then $|\mathcal P|>s(G)$.  If $H$ is the subgroup of $G$ guaranteed by Theorem \ref{thm} (applied using $\mathcal P$ as our $A$) then there exists an element $h\in H$ such that
\[h\equiv 1 \pmod L,\]
\[h\equiv a \pmod M.\]
Equivalently, there exists a subset of $\mathcal P$ whose product is a Carmichael number that is congruent to $a$ (mod $M$).
\end{theorem}

\begin{proof}
The fact that $|\mathcal P|>s(G)$ is obvious from the computation of both of these quantities in Lemmas \ref{l512} and \ref{lemfour2}, respectively.  So $\mathcal P\cap H$ is non-empty.  Thus, let $p_H$ be a prime such that $p_H\in \mathcal P\cap H$.

By construction, $p_H$ is a quadratic residue mod $L$.  So $ord_L(p_H)|\prod_{q|L}\frac{\phi(q)}{2}$.  Since $L$ is the product of primes that are 3 mod 4, this implies that the order of $p_H$ is odd modulo $L$.  By assumption, $(\frac{\phi(q)}{2}, \phi(M))=1$ for every prime $q|L$, which means that there exists an integer $r$ such that
\[r\equiv 0 \pmod{ord_L(p_H)},\]
\[r\equiv 1 \pmod{\phi(M)}.\]
Since $p_H\in H$ and $H$ is a group, $p_H^r$ must also be in $H$. But
\[p_H^r\equiv 1 \pmod L,\]
\[p_H^r\equiv a \pmod M,\]
which gives us the required congruence conditions.

Now, by Theorem \ref{thm}, the above congruences imply that there exists a square-free $n$, composed of primes in $\mathcal P$, such that
\[n\equiv 1 \pmod L,\]
\[n\equiv a \pmod M.\]
Since $n$ is a product of primes that are all 1 modulo $k$, it follows that $n$, too, is 1 mod $k$.  So for any prime $p|n$, we know that $p-1|kL|n-1$.  Thus, by Korselt's Criterion, $n$ is a Carmichael number that is congruent to $a$ modulo $M$.
\end{proof}

Let us now use the second part of Theorem \ref{thm} to prove the asymptotic.  We must choose a $t$ such that $s(G)<t < n -n(G)$.  As such, we will let
\[t=\frac{\left(\frac 65\right)^{\omega(L)}}{60\cdot \phi(M)\log x}.\]
\begin{varthm1}
Let $C_{M,a}(X)$ be the number of Carmichael numbers up to $X$ that are congruent to $a$ mod $M$.  Then there exists a constant $K>0$ such that
\[C_{M,a}(X)\gg \left(X\right)^\frac{K}{(\log\log\log X)^2}.\]
\end{varthm1}
\begin{proof}
Clearly, $s(G)<t<|\mathcal P|-n(G)$.  So the number of products of at most $t$ primes in $|\mathcal P|$ whose product is 1 mod $L$ and $a$ mod $M$ is at least
\[\left(\begin{array}{c} |\mathcal P|-n(G)\\ t-n(G)\end{array}\right)/\left(\begin{array}{c} |\mathcal P|\\n(G)\end{array}\right).\]
Of course, we recall from Lemma \ref{lemfour2} that $n(G)$ is much smaller than both $\mathcal P$ and $t$.  So it is clear that
\[|\mathcal P|-n(G)\geq \frac 45|\mathcal P|,\label{eq1}\]
and
\[t-n(G)\geq \frac 23 t.\label{eq2}\]
We will also use the standard bound that
\[\left(\frac{u}{v}\right)^v\leq \left(\begin{array}{c} u\\v\end{array}\right)\leq \left(\frac{ue}{v}\right)^v.\label{eq3}\]
Putting these four statements together,
\begin{align*}
\left(\begin{array}{c} |\mathcal P|-n(G)\\ t-n(G)\end{array}\right)/& \left(\begin{array}{c} |\mathcal P|\\n(G)\end{array}\right)\\
\geq & \left(\begin{array}{c} |\mathcal P|-n(G)\\ t-n(G)\end{array}\right)/ \left(\begin{array}{c} |\mathcal P|\\ \frac t3\end{array}\right)\\
\geq &\left(\frac{\frac 45 |\mathcal P|}{t}\right)^{\frac 23 t}/\left(\frac{3e|\mathcal P|}{t}\right)^{\frac t3}\\
\geq &\left(\frac{16}{75e}\right)^{\frac 13 t}\left(\frac{|\mathcal P|}{t}\right)^{\frac 13 t}\\
\end{align*}
By Lemma \ref{l512} and the definition of $t$, we have
\begin{align*}
\left(\frac{16}{75e}\right)^{\frac 13 t}&\left(\frac{|\mathcal P|}{t}\right)^{\frac 13 t}\\
\geq &\left(\frac{16\cdot 15}{75e}\right)^{\frac 13 t}\left(\frac{\left(\frac 32\right)^{\omega(L)}}{\left(\frac 65\right)^{\omega(L)}}\right)^{\frac 13 t}\\
\geq &\left(\left(\frac 54\right)^{\omega(L)}\right)^{\frac 13 t}.\\
\end{align*}
Recall from Lemma \ref{lemtwo} that
\[\omega(L)\geq \gamma\frac{y^{\theta}}{\log y^\theta}.\]
Define
\[X=x^t.\]
Using the bound for $\omega(L)$ above and the definition of $t$, we have that for $y$ sufficiently large,
\[(\log\log\log X)^2\geq \log y^\theta.\]
Moreover, from Lemma \ref{lemfour}, we have
\[X\leq (e^{\frac{2\kappa_2}{B} y^\theta})^{2t}.\]
Note also that any Carmichael number that is the product of at most $t$ primes in $\mathcal P$ must be less than $X$.  So the number derived from the combinatorial functions above is a lower bound for $C_{M,a}(X)$.

As such, we have that, for sufficiently large values of $y$,
\begin{align*}
C_{M,a}(X)\geq &\left(\left(\frac 54\right)^{\frac 13\omega(L)}\right)^{t}\\
\geq & \left(\left(e^{\frac{\kappa_1 y^\theta}{\log y^\theta}}\right)^t\right)^{\frac{\gamma\log(1.2)}{3}}\\
\geq & \left(X\right)^\frac{\gamma\log(1.2)B\kappa_1}{12\kappa_2(\log\log\log X)^2},
\end{align*}
where the second line is again from Lemma \ref{lemfour}.  The theorem then follows.
\end{proof}

\section{Acknowledgements}
Thanks to Carl Pomerance and Steven J. Miller for the many helpful comments and suggestions.  Thanks also to the referee for helping to make this paper far more readable.

\end{document}